\documentclass{article}[12pt]

\usepackage{fullpage}
\usepackage{amssymb}
\usepackage{amsmath}
\usepackage{amsthm}
\usepackage{url}
\usepackage{hyperref}
\usepackage{mathabx,rotating}
\usepackage{multirow}
\usepackage{mathabx}
\usepackage{tablefootnote}
\usepackage{pdflscape}

\newtheorem{theorem}{Theorem}
\newtheorem{corollary}{Corollary}

\newtheorem{lemma}{Lemma}

\newtheorem{definition}{Definition}

\begin{document}

\title{On rearrangement inequalities for T-norm logics}
\date{April 12, 2022\\Latest update: July 25, 2023}
\author{Chai Wah Wu\\MIT-IBM Watson AI Lab\\IBM T. J. Watson Research Center\\P. O. Box 218, Yorktown Heights, New York 10598, USA\\e-mail: chaiwahwu@ieee.org}

\maketitle

\begin{abstract}
The rearrangement inequality states that the sum of products of permutations of 2 sequences of real numbers are maximized when the terms are similarly ordered and minimized when the terms are ordered in opposite order. We show that similar inequalities exist for multi-valued logic with the multiplication and addition operation replaced with various  $T$-norms and $T$-conorms respectively. For instance, we show that the rearrangement inequality holds when the $T$-norms and $T$-conorms are derived from Archimedean copulas.
\end{abstract}

\section{Introduction}
The well-known rearrangement inequality \cite{Hardy1952} plays a key role in the derivation of many mathematical inequalities. It states that for two finite sequences of $n$ real numbers, the sum of the product of pairs of terms is maximal when the sequences are similarly ordered and minimal when oppositely ordered. In other words, suppose $x_1 \leq x_2 \leq \cdots \leq x_n$  
and $y_1 \leq y_2 \leq \cdots \leq y_n$ are real numbers, then for any permutation $\sigma$ in the symmetric group $S_n$ of permutations on $\{1,\cdots , n\}$,
\begin{equation}x_ny_1+\cdots +x_1y_n \leq x_{\sigma(1)}y_1+\cdots + x_{\sigma(n)}y_n \leq x_1y_1+\cdots x_ny_n
\label{eqn:rearrange1}
\end{equation}

The dual inequality, obtained by swapping addition with multiplication and the direction of the inequality, is also true \cite{oppenheim:rearrangement:1954}, but only for nonnegative numbers in general ($x_i\geq 0$, $y_i\geq 0$):

\begin{equation}
(x_1+y_1)\times \cdots \times (x_n+y_n) \leq (x_{\sigma(1)}+y_1)\times\cdots \times (x_{\sigma(n)}+y_n ) \leq (x_n+y_1)\times\cdots \times (x_1+y_n)
\label{eqn:rearrange2}
\end{equation}

These two inequalities say that similarly ordered terms minimize the product of sums of pairs, while opposite ordered terms maximize the product of sums.
In Ref. \cite{Minc1971} it was shown that Eq. (\ref{eqn:rearrange1}) and Eq. (\ref{eqn:rearrange2}) are equivalent for positive numbers.

In Eqns. (\ref{eqn:rearrange1}-\ref{eqn:rearrange2}) there are two binary operations $+$ and $\times$ and an order $\leq$ and in Ref. \cite{wu:rearrangement:2022} this has been extended to partially ordered rings (po-rings) such as Hermitian matrices.
The goal of this paper is to extend these inequalities to algebraic structures governing multi-valued logic. The motivation here is that there are multi-valued logics where the conjunction is defined as multiplication and the disjunction is defined as addition clipped at $1$. It is therefore of interest to study whether consequences of arithmetical operations, such as the rearrangement inequality, extend to more general definitions of conjunction and disjunction, and what the requirements are that allow for such extensions.

\section{Multi-valued logic via $T$-norms}

Note that if $x_i$, $y_i$ are variables in Boolean algebra, with $+$ and $\times$ representing the Boolean logical disjunction $\vee$ and conjunction $\wedge$ respectively, then Eqns. (\ref{eqn:rearrange1}-\ref{eqn:rearrange2}) are trivially true. Classical Boolean logic have been extended to logic systems that can take more than 2 values in many different ways \cite{sep-logic-manyvalued}. We consider such algebras of logic and replace $+$ and $\times$ in the rearrangement inequality with disjunction and conjunction respectively. 
One approach to multi-valued logic is by using $T$-norms (or triangular norms) and $T$-conorms to denote conjunction and disjunction respectively. Such norms have been studied intensely in the past and generalize well-known multi-valued logic such as \L{}ukasiewicz logic. We summarize in this section the standard definitions and results of $T$-norms and $T$-conorms \cite{fodor:t-norm:2004,KLEMENT:positonpaperI:2004,JENEI200427}.

\begin{definition} \label{def:t-norm}
A function $\otimes:[0,1]^2\rightarrow [0,1]$ is called a $T$-norm if
\begin{itemize}
\item $\otimes$ is commutative
\item $\otimes$ is associative
\item $\otimes$ is monotonic, i.e. for all $x,y,z\in [0,1]$, $x\leq y$ implies that
$x\otimes z \leq y\otimes z$.
\item Neutrality of $1$, i.e. $x\otimes 1 = x$ for all $x\in [0,1]$.
\end{itemize}
\end{definition}

\begin{definition} \label{def:t-conorm}
A function $\oplus:[0,1]^2\rightarrow [0,1]$ is called a $T$-conorm (or an $S$-norm) if
\begin{itemize}
\item $\oplus$ is commutative
\item $\oplus$ is associative
\item $\oplus$ is monotonic
\item Neutrality of $0$, i.e. $x\oplus 0 = x$ for all $x\in [0,1]$.
\end{itemize}
\end{definition}

\begin{definition}
For a $T$-norm $f$, a number $0<x<1$ is called a zero divisor of $f$ if there exists a number $0<y<1$ such that $f(x,y) = 0$.
\end{definition}

\begin{definition}
A monotonically nonincreasing function $\neg: [0,1]\rightarrow [0,1]$ such that $\neg (0) = 1$ and $\neg (1) = 0$ is called a {\em negation}. It is called {\em strict} if it is strictly monotone and it is called {\em strong} if it is strict and involutive, i.e.
$\neg \neg x = x$ for all $x \in [0,1]$.
\end{definition}

A commonly used (strong) negation function is $x \rightarrow 1-x$. All strong negation functions can be written as $\phi^{-1}(1-\phi(x))$ for some continuous strictly increasing bijection $\phi$ on $[0,1]$ \cite{trillas:negation:1979}.

\begin{definition}
For a function $f:[0,1]^2\rightarrow [0,1]$ and a negation function $\neg$, its {\em dual} $g = \Phi^{\neg} (f)$ is defined as
$g(x,y) = \neg(f(\neg(x),\neg(y)))$.
\end{definition}

It can be shown that for a strong negation functions $\neg$, 
$\Phi^{\neg} (\Phi^{\neg} (f)) = f$ and
$f$ is a $T$-norm if and only if its dual $\Phi^{\neg}(f)$ is a $T$-conorm. 
If $\neg$ is the standard negation function $x \rightarrow 1-x$ we will simply use $\Phi$ to denote $\Phi^{\neg}$, i.e. $\Phi(f)(x,y) = 1-f(1-x,1-y)$.

Examples of well-known $T$-norms and $T$-conorms are listed in Tables \ref{tbl:t_norm}-\ref{tbl:t_conorm} where each $T$-norm in Table \ref{tbl:t_norm} is the dual (via the standard negation) of the corresponding $T$-conorm in Table \ref{tbl:t_conorm} \cite{dubois1980fuzzy,Mayor1991,Klement:positionII:2004,nelsen:copulas:2006}. 

\begin{table}[htbp]
\centering
\begin{tabular}{|c||c|c|c|}
\hline
& $T$-norm & \shortstack{property\\$A$} &\shortstack{property\\$B$}\\
\hline\hline
minimum ({G\"odel}) $T_m$ & $\min (x,y)$ &\checkmark & \checkmark \\ \hline
product $T_p$ & $xy$ &\checkmark & \checkmark  \\  \hline
\L{}ukasiewicz $T_L$ & $\max(x+y-1,0)$ &\checkmark & \checkmark \\  \hline
drastic $T_d$  & $\min (x,y) $ if $\max(x,y) = 1$, otherwise $0$&&  \\  \hline
nilpotent minimum $T_n$ & $\min(x,y)$ if $x+y > 1$, otherwise $0$&\checkmark&  \\  \hline
Dubois-Prade $T_{DP}^{\alpha}$& $\frac{xy}{\max(x,y,\alpha)}$ for some $\alpha\in [0,1]$ &\checkmark&\checkmark\\\hline
Ali-Mikhail-Haq $T_{AMH}^\alpha$ & $\begin{array}{cl} \frac{xy}{\alpha+(1-\alpha)(x+y-xy)} & \mbox{if } \max(x,y) > 0 \\
0 & \mbox{otherwise}\end{array}$, $\alpha \in [0,2]$&\checkmark&\checkmark  \\  \hline
Clayton $T_C^\alpha$ & $\max(x^\alpha+y^\alpha-1,0)^{\frac{1}{\alpha}}$, $0 \neq \alpha \leq 1$&\checkmark&\checkmark  \\  \hline
Frank $T_F^\alpha$ & $\log_{\alpha}\left(1+\frac{(\alpha^x-1)(\alpha^y-1)}{\alpha-1}\right)$, $0 < \alpha \neq 1$&\checkmark&\checkmark  \\  \hline
Yager $T_Y^\alpha$ & $\max(1-((1-x)^\alpha + (1-y)^\alpha)^\frac{1}{\alpha},0)$, $\alpha \geq 1$&\checkmark&\checkmark  \\  \hline
Mayor-Torrens $T_{MT}^\alpha$ & $\begin{array}{cl}\max(x+y-\alpha,0) & \mbox{if }  0\leq x,y \leq \alpha\\
                                                                                    \min(x,y) & \mbox{otherwise} \end{array}$, $\alpha\in [0,1]$ &\checkmark&\checkmark  \\  \hline
Sugeno-Weber\tablefootnote{The Sugeno-Weber $T$-norms are typically defined for all $\alpha \geq -1$. We consider only the case $\alpha \geq 0$ as it allows us to prove properties $A$ and $B$.} $T_{SW}^\alpha$ & $\max(\frac{x+y-1+\alpha xy}{1+\alpha},0)$, $\alpha \geq 0$&\checkmark&\checkmark  \\  \hline
Gumbel $T_{G}^\alpha$ & $\begin{array}{cl} e^{-((-\ln(x))^\alpha+(-\ln(y))^\alpha)^\frac{1}{\alpha}} & \mbox{if } \min(x,y) > 0 \\ 0 & \mbox{otherwise}
\end{array}$, $\alpha \geq 1$&\checkmark&\checkmark  \\  \hline
Joe $T_{J}^\alpha$ & $1-\left((1-x)^\alpha+(1-y)^\alpha-(1-x)^\alpha (1-y)^\alpha\right)^{\frac{1}{\alpha}}$, $\alpha \geq 1$ &\checkmark&\checkmark  \\ \hline
\end{tabular}
\caption{Table of $T$-norms. Each $T$-norm is a dual via the standard negation to a corresponding $T$-conorm in Table \ref{tbl:t_conorm}.}
\label{tbl:t_norm}
\end{table}

\begin{table}[htbp]
\centering
\begin{tabular}{|c||c|c|c|}
\hline
& $T$-conorm & \shortstack{property\\$A'$} & \shortstack{property\\$B'$}\\
\hline\hline
maximum $T_m'$ & $\max (x,y)$ &\checkmark & \checkmark  \\ \hline
\shortstack{probabilistic\\sum $T_p'$}& $x+ y-xy$ &\checkmark & \checkmark \\  \hline
\shortstack{bounded\\sum $T_L'$} & $\min(x+y,1)$ &\checkmark & \checkmark\\  \hline
drastic $T_d'$ & $\max (x,y)$ if $\min(x,y) = 0$, otherwise $1$&&  \\  \hline
\shortstack{nilpotent\\maximum $T_n'$} & $\max(x,y)$ if $x+y < 1$, otherwise $1$ &\checkmark & \\  \hline
\shortstack{Dubois-Prade\\$T_{DP}'^{\alpha}$}& $1- \frac{(1-x)(1-y)}{1-\min(x,y,1-\alpha)}$ for some $\alpha\in [0,1]$ &\checkmark&\checkmark\\\hline
\shortstack{Ali-Mikhail-Haq\\ $T_{AMH}'^\alpha$} & $\begin{array}{cc} \frac{x+y+(\alpha -2)xy}{1+(\alpha -1)xy} & \mbox{if } \min(x,y) < 1 \\
1 & \mbox{otherwise}\end{array}$, $\alpha \in [0,2]$&\checkmark&\checkmark  \\  \hline
Clayton $T_C'^\alpha$ & $1-\max((1-x)^\alpha+(1-y)^\alpha-1,0)^{\frac{1}{\alpha}}$, $0\neq \alpha \leq 1$&\checkmark&\checkmark  \\  \hline
Frank $T_F'^\alpha$ & $1-\log_{\alpha}\left(1+\frac{(\alpha^{1-x}-1)(\alpha^{1-y}-1)}{\alpha-1}\right)$, $\alpha\neq 0, 1$&\checkmark&\checkmark  \\  \hline
Yager $T_Y'^\alpha$ & $\min((x^\alpha + y^\alpha)^\frac{1}{\alpha},1)$, $\alpha \geq 1$&\checkmark&\checkmark  \\  \hline
\shortstack{Mayor-Torrens\\$T_{MT}'^{\alpha}$} & $\begin{array}{cl}\min(x+y+\alpha-1,1) & \mbox{if }  1-\alpha \leq x,y \leq 1\\
                                                                                    \max(x,y) & \mbox{otherwise} \end{array}$, $\alpha\in [0,1]$ &\checkmark&\checkmark  \\  \hline
\shortstack{Sugeno-Weber\\$T_{SW}'^\alpha$} & $\min (x+y-\frac{\alpha}{1+\alpha} xy,1)$, $\alpha \geq 0$&\checkmark&\checkmark  \\  \hline
Gumbel $T_{G}'^\alpha$ & $\begin{array}{cl} 1-e^{-((-\ln(1-x))^\alpha+(-\ln(1-y))^\alpha)^\frac{1}{\alpha}} & \mbox{if } \max(x,y) < 1 \\ 1 & \mbox{otherwise}\end{array}$, $\alpha \geq 1$&\checkmark&\checkmark  \\  \hline
Joe $T_{J}'^\alpha$ & $\left(x^\alpha+y^\alpha-x^\alpha y^\alpha\right)^{\frac{1}{\alpha}}$, $\alpha \geq 1$ &\checkmark&\checkmark  \\ \hline
\end{tabular}
\caption{Table of $T$-conorms.  Each $T$-conorm is a dual via the standard negation to a corresponding $T$-norm in Table \ref{tbl:t_norm}.}
\label{tbl:t_conorm}
\end{table}

The procedure of ordinal sum produces another $T$-norm from a countable sequence of $T$-norms and is defined as \cite{clifford:semigroups:1954,JENEI2002199}:
\begin{definition}
Let $(a_i,b_i)\subseteq [0,1]$ be a family of pairwise disjoint and non-empty open intervals and $f_i$ be $T$-norms.
The ordinal sum $g = \Asterisk_{i} (f_i,a_i,b_i)$ is defined as:
$$g(x,y) = \left\{\begin{array}{ll} a_i + (b_i-a_i)f_i\left(\frac{x-a_i}{b_i-a_i},\frac{y-a_i}{b_i-a_i}\right) & \mbox{if } x,y\in [a_i,b_i]^2 \\ \min(x,y) & \mbox{otherwise}
\end{array}\right.
$$
\end{definition}

It has been shown that the ordinal sum of $T$-norms is a $T$-norm.
Another general method of constructing $T$-norms is via generators.

\begin{definition}[\cite{Klement:positionII:2004}] 
For a $T$-norm $f$, a function $\mu:[0,1]\rightarrow [0,\infty]$ is called an {\em additive generator} of $f$ if $\mu$ is strictly decreasing, right-continuous at $0$, $\mu(1) = 0$ such for all $0\leq x, y \leq 1$, $\mu(x)+\mu(y) \in \mbox{Range}(\mu) \cup [\mu(0),\infty]$ and 
$f(x,y) = \mu^{(-1)}(\mu(x)+\mu(y))$ where $\mu^{(-1)}$ is the pseudo-inverse of $\mu$.
\end{definition}

Not all $T$-norms have an additive generator. 

\begin{definition}
A $T$-norm $\otimes$ is called {\em Archimedean} if for all $0<x, y<1$, there exists $n$ such that $x^{\otimes n}$ (i.e. $x\otimes...\otimes x$ $n$ times) $\leq y$.
\end{definition} 

\begin{theorem}[\cite{Klement:positionII:2004}]
If a $T$-norm has an additive generator, then it is Archimedean.
\end{theorem}

The definitions of $T$-norm and $T$-conorm are almost identical and differ only in the identity element. In \cite{YAGER1996111,fodor:uninorm:1997} these definitions are unified by allowing an arbitrary identity element.

\begin{definition}\label{def:uninorm}
A function $\otimes:[0,1]^2\rightarrow [0,1]$ is called a uninorm if
\begin{itemize}
\item $\otimes$ is commutative
\item $\otimes$ is associative
\item $\otimes$ is monotonic
\item Existence of identity element $e$, i.e. there exists $e\in [0,1]$ such that $x\otimes e = x$ for all $x\in [0,1]$.
\end{itemize}
\end{definition}
It is clear that a $T$-norm and a $T$-conorm are uninorms with identity element $1$ and $0$ respectively.

\begin{definition}
A uninorm $\otimes$ is {\em conjunctive} if $0\otimes1 = 0$ and is {\em disjunctive} if $0\otimes 1 = 1$.
\end{definition}
It was shown that a uninorm is either conjunctive or disjunctive \cite{fodor:uninorm:1997} and it is clear that a $T$-norm is a conjunctive uninorm and a $T$-conorm is a disjunctive uninorm.
Note that all the $T$-norms and $T$-conorms in Tables \ref{tbl:t_norm}-\ref{tbl:t_conorm} are conjunctive and disjunctive respectively.

\section{Rearrangement inequalities for $T$-norms}

\begin{definition}
Let $\otimes$ and $\oplus$ be uninorms. We say that $(\otimes, \oplus)$ satisfies the {\em rearrangement inequality} if for all $n$ and for all 
$0\leq x_1 \leq x_2 \cdots \leq x_n\leq 1$ 
and $0\leq y_1 \leq y_2 \cdots \leq y_n \leq 1$, and any permutation $\sigma\in S_n$,
\begin{equation}(x_n\otimes y_1)\oplus\cdots \oplus (x_1\otimes y_n) \leq (x_{\sigma(1)}\otimes y_1)\oplus \cdots \oplus  (x_{\sigma(n)}\otimes y_n) \leq (x_1\otimes y_1)\oplus \cdots \oplus (x_n\otimes y_n)
\label{eqn:rearrange1_t}
\end{equation}
We say that $(\otimes, \oplus)$ satisfies the {\em dual rearrangement inequality} if for all $n$ and for all 
$0\leq x_1 \leq x_2 \cdots \leq x_n\leq 1$ 
and $0\leq y_1 \leq y_2 \cdots \leq y_n \leq 1$, and any permutation $\sigma\in S_n$,
\begin{equation}(x_n\oplus y_1)\otimes\cdots \otimes (x_1\oplus y_n) \geq (x_{\sigma(1)}\oplus y_1)\otimes \cdots \otimes  (x_{\sigma(n)}\oplus y_n) \geq (x_1\oplus y_1)\otimes \cdots \otimes (x_n\oplus y_n)
\label{eqn:rearrange2_t}
\end{equation}
\end{definition}

The associativity and commutativity implies that $(\otimes, \otimes)$ trivially satisfies the rearrangement inequality and its dual.
Not all choices for $(\otimes, \oplus)$ will satisfy the rearrangement inequality or its dual. For instance, for $\otimes = T_n$ and $\oplus=T_L'$, the rearrangement inequality is not satisfied as evidenced by the two $2$-element sequences $x_1=0.2, x_2=0.7$ and $y_1 = 0.6, y_2=0.9$. Similarly, for $\otimes = T_d$ and $\oplus=T_L'$,
the dual rearrangement inequality is not satisfied as evidenced by the two $2$-element sequences $x_1=0.38, x_2=0.96$ and $y_1 = 0.005, y_2=0.05$. 

Our first result gives necessary and sufficient conditions for a pair $(\otimes, \oplus)$ to satisfy the rearrangement inequality and its dual.

\begin{theorem}
Let $\otimes$ and $\oplus$ be uninorms. $(\otimes, \oplus)$ satisfies the rearrangement inequality if and only if for all $0\leq x_1\leq x_2\leq 1$ and $0\leq y_1\leq y_2\leq 1 $,
\begin{equation} (x_1\otimes y_1)\oplus (x_2\otimes y_2) \geq (x_1\otimes y_2)\oplus (x_2\otimes y_1) \label{eqn:condition} \end{equation}
 $(\otimes, \oplus)$ satisfies the dual rearrangement inequality if and only if for all $0\leq x_1\leq x_2\leq 1$ and $0\leq y_1\leq y_2\leq 1$, 
\begin{equation} (x_1\oplus y_1)\otimes (x_2\oplus y_2) \leq (x_1\oplus y_2)\otimes (x_2\oplus y_1) \label{eqn:condition_dual} \end{equation}
\end{theorem}

\begin{proof}
One direction is trivially true. For the other direction, the proof is a consequence of Theorem 8 in Ref. \cite{wu:rearrangement:2022} and the commutative, associative, and monotonicity properties of uninorms along with the nonnegativity of the unit interval domain.
\end{proof}

\begin{theorem}\label{thm:negation_RI}
Let $\otimes$ and $\oplus$ be uninorms and $\neg$ be a strong negation function.
Then $(\otimes, \oplus)$ satisfies the rearrangement inequality if and only if
$(\Phi^{\neg}(\oplus),\Phi^{\neg}(\otimes))$ satisfies the dual rearrangement inequality.
\end{theorem} 
\begin{proof}
Let
$0\leq x_1\leq x_2\leq 1$ and $0\leq y_1\leq y_2\leq 1$. Let $\tilde{x}_i = \neg x_i$ and $\tilde{y}_i = \neg y_i$. Then $0\leq \tilde{x}_2\leq \tilde{x}_1\leq 1$ and $0\leq \tilde{y}_2\leq \tilde{y}_1\leq 1$. Let $\tilde{\oplus} = \Phi^{\neg}(\otimes)$ and $\tilde{\otimes} = \Phi^{\neg}(\oplus)$.
Then
we have the following:
$(x_1\otimes y_1)\oplus (x_2\otimes y_2) = \neg (\tilde{x}_1\tilde{\oplus} \tilde{y}_1)\oplus \neg (\tilde{x}_2\tilde{\oplus} \tilde{y}_2) = \neg ((\tilde{x}_1\tilde{\oplus} \tilde{y}_1)\tilde{\otimes} (\tilde{x}_2\tilde{\oplus} \tilde{y}_2))$. Similarly,
$(x_1\otimes y_2)\oplus (x_2\otimes y_1) =  \neg ((\tilde{x}_1\tilde{\oplus} \tilde{y}_2)\tilde{\otimes} (\tilde{x}_2\tilde{\oplus} \tilde{y}_1))$.
Eqn.~(\ref{eqn:condition}) is then equivalent to
$(\tilde{x}_1\tilde{\oplus} \tilde{y}_1)\tilde{\otimes} (\tilde{x}_2\tilde{\oplus} \tilde{y}_2) \leq  (\tilde{x}_1\tilde{\oplus} \tilde{y}_2)\tilde{\otimes} (\tilde{x}_2\tilde{\oplus} \tilde{y}_1)$ which is equivalent to $(\tilde{\otimes},\tilde{\oplus})$ satisfying the dual rearrangement inequality.
\end{proof}

\begin{corollary}
Let $\otimes$ be a uninorm and $\neg$ be a strong negation function. Then
$(\otimes, \Phi^{\neg}(\otimes))$ satisfies the rearrangement inequality if and only if $(\otimes, \Phi^{\neg}(\otimes))$ satisfies the dual rearrangement inequality.
\end{corollary}

We first simplify the analysis of Eqns.~(\ref{eqn:condition}-\ref{eqn:condition_dual}) by decomposing it into two steps. First we give conditions under which they hold when $\oplus$ is replaced with $+$. We then give conditions under which this implies that Eqns.~(\ref{eqn:condition}-\ref{eqn:condition_dual}) also hold.

\begin{definition} \label{def:propA}
A function $f:[0,1]^2\rightarrow [0,1]$ satisfies property $A$ if for all $0\leq x\leq y\leq z\leq w\leq 1$
\begin{equation}
w+x\leq y+z\Rightarrow f(x,w) \leq  f(y,z) 
\label{eqn:propA}
\end{equation}

$f$ satisfies property $A'$ if for all $0\leq x\leq y\leq z\leq w\leq 1$
\begin{equation}
w+x\geq y+z\Rightarrow  f(x,w) \geq  f(y,z) 
\label{eqn:propA'}
\end{equation}
\end{definition}

\begin{definition} \label{def:propB}
A function $f:[0,1]^2\rightarrow [0,1]$ satisfies property $B$ if for all $0\leq x\leq y \leq 1$ and $0\leq z\leq w\leq 1$
\begin{equation}
f(x,w) - f(x,z) \leq f(y,w)-f(y,z)
\label{eqn:propB}
\end{equation}

$f$ satisfies property $B'$ if for all $0\leq x\leq y \leq 1$ and $0 \leq z\leq w\leq 1$
\begin{equation}
f(x,w) - f(x,z) \geq f(y,w)-f(y,z)
\label{eqn:propB'}
\end{equation}
\end{definition}

It is clear that if $f$, $g$ both satisfy the same property in Defns \ref{def:propA}-\ref{def:propB}, then that property is satisfied by $\alpha f+(1-\alpha)g$ as well for all  $\alpha \in [0,1]$.

Property $B$ is also called $2$-increasing and is used in the definition of copulas \cite{janssens:t-norm:2004}.
\begin{definition}
A function $f:[0,1]^2\rightarrow [0,1]$ is called a {\em copula} if it satisfies neutrality of $1$, is monotonic and satisfies property $B$. 
\end{definition}
Property $B$ implies that $f$ is uniformly continuous.
For instance, $\min$, product and \L{}ukasiewicz are copulas. Furthermore, all copulas are bounded pointwise by \L{}ukasiewicz and $\min$. Some families of $T$-norms that are copulas can be found in Ref. \cite{janssens:t-norm:2004}. 

\begin{theorem}[\cite{moynihan:copula:1978,nelsen:copulas:2006,BACIGAL:additive:2015}]
A $T$-norm is an Archimedean copula if and only if it has a convex additive generator.
\label{thm:arch_copula}
\end{theorem}

\begin{lemma}
If all $f_i$ satisfies property $A$, then the ordinal sum $\Asterisk_{i} (f_i,a_i,b_i)$ satisfies property $A$. If all $f_i$ satisfies property $B$, then $\Asterisk_{i} (f_i,a_i,b_i)$ satisfies property $B$.
\label{lem:ordinal_sum} 
\end{lemma}

\begin{proof}
With regards to property $B$, the proof can be found in Ref. \cite{Mesiar:ordinalsums:2010} as the ordinal sum of copulas is a copula.
For property $A$ the proof is similar. First consider the case of $2$ intervals.
Let $f_1$ and $f_2$ be $T$-norms satisfying property $A$. Let $f$ be the ordinal sum of
$(f_1,(0,a))$ and $(f_2,(a,1))$. Let
$0\leq x\leq y \leq z\leq w\leq 1$. $w+x\leq y+z$ implies that  $\frac{w-a}{1-a}+\frac{x-a}{1-a}\leq\frac{y-a}{1-a}+\frac{z-a}{1-a}$.
If $x\geq a$, then $f(x,w)-f(y,z) = (1-a)(f_2(\frac{x-a}{1-a},\frac{w-a}{1-a})-f_2(\frac{y-a}{1-a},\frac{z-a}{1-a}))$ which is nonpositive since $f_2$ satisfies property $A$. If $w\leq a$, then $f(x,w)-f(y,z) = a\left(f_1(\frac{x}{a},\frac{w}{a})-f_1(\frac{y}{a},\frac{z}{a})\right)$ which again is nonpositive since $f_1$ satisfies property $A$.
If $x\leq a\leq y$, then $f(x,w) = \min(x,w) = x$ and $f(y,z) \geq a$ and thus  $f(x,w)-f(y,z) \leq x-a\leq 0$.
If $z\leq a\leq w$, then $f(x,w) = \min(x,w) = x$ and $a+x\leq w+x\leq y+z$, i.e. $1+\frac{x}{a}\leq \frac{y}{a}+\frac{z}{a}$.
$f_1$ satisfying property $A$ implies that $\frac{x}{a} = f_1(1,\frac{x}{a}) \leq f_1(\frac{y}{a},\frac{z}{a}) = \frac{1}{a}f(y,z)$. Thus $f(x,w)\leq f(y,z)$.
If $y\leq a \leq z$, then $f(x,w)-f(y,z) = \min(x,w)-\min(y,z) = x-y\leq 0$. Thus in all cases, Eq. (\ref{eqn:propA}) is satisfied.
The case of a finite number of terms in the ordinal sum then follows from induction.
For the case of countable infinite number of terms, the same approach as in Ref. \cite{Mesiar:ordinalsums:2010} can be used to construct a sequence of ordinal sums of finite terms that converges pointwise to the desired ordinal sums and Eq. (\ref{eqn:propA}) is preserved.
\end{proof}

\begin{lemma} \label{lem:propAB_dual}
$f$ satisfies property $A$ if and only if $\Phi(f)$ satisfies property $A'$, $f$ satisfies property $B$ if and only if $\Phi(f)$ satisfies property $B'$.
\end{lemma}
\begin{proof}
Suppose $f$ satisfies property $A$. Let $0\leq x\leq y\leq z\leq w\leq 1$ and let $x'= 1-x$, $y'=1-y$, $z'=1-z$, $w' = 1-w$.
Then $0\leq w'\leq z'\leq y'\leq x'\leq 1$. Suppose $w+x \geq y+z$, then $w'+x' \leq y'+z'$. Since property $A$ implies $f(y',z') \geq f(x',w')$, $\Phi(f)(y,z) = 1-f(y',z') \leq 1-f(x',w') = \Phi(f)(x,w)$, and $\Phi(f)$ satisfies property $A'$. An analogous argument show that if $\Phi(f)$ satisfies property $A'$, then $f$ satisfies property $A$.

Suppose $f$ satisfies property $B$. Let $0\leq x\leq y \leq 1$ and $0\leq z\leq w\leq 1$. Thus $0\leq y'\leq x' \leq 1$ and $0\leq w' \leq z' \leq 1$.
Property $B$ implies $f(y',z')-f(y',w') \leq f(x',z')-f(x',w')$. This implies that $-\Phi(f)(y,z)+\Phi(f)(y,w) \leq -\Phi(f)(x,z)+\Phi(f)(x,w)$, i.e. $\Phi(f)$ satisfies property $B'$. The other direction follows analogously.
\end{proof}

\begin{theorem} \label{thm:gen_condition}
If $f$ is a $T$-norm with an additive generator $\mu$ such that
for all $0\leq x\leq y\leq 1$, and  $h\geq 0$ such that $y+h\leq 1$, we have $\mu(x)-\mu(x+h)\geq \mu(y)-\mu(y+h)$, then
$f$ satisfies property $A$.
\end{theorem}

\begin{proof}
Let $0\leq x\leq y\leq z\leq w\leq 1$ and
$y-x\geq w-z$. Define $x' = y-w+z \geq x$. The condition on $\mu$ shows that $\mu(x')-\mu(y) \geq \mu(z)-\mu(w)$ and
$\mu(x') + \mu(w) \geq \mu(y) +\mu(z)$. Since $\mu^{(-1)}$ is nonincreasing \cite{KLEMENT:pseudo-inverse:1999}, $f(x,w)\leq f(x',w) = \mu^{(-1)}(\mu(x') + \mu(w))\leq  \mu^{(-1)}(\mu(y) + \mu(z)) = f(y,z)$.
\end{proof}

\begin{corollary}
If $f$ is a $T$-norm with an differentiable additive generator $\mu$ such that
$\frac{d\mu(x)}{dx}$ is nondecreasing, then $f$ satisfies property $A$.
\label{cor:addgenerator}
\end{corollary}

\begin{corollary}
Let $f$ be a $T$-norm. If $f$ is an Archimedean copula, then $f$ satisfies both property $A$ and $B$.
\label{cor:arch_copula}
\end{corollary}
\begin{proof}
Since the additive generator $\mu$ of an Archimedean copula is convex and strictly decreasing by Theorem \ref{thm:arch_copula},
$\mu(x+h) \leq \lambda \mu(x) + (1-\lambda)(\mu(y+h))$ where $\lambda = \frac{y-x}{y-x+h}$.
Similarly,
$\mu(y) \leq \kappa \mu(x) + (1-\kappa)(\mu(y+h))$ where $\kappa = \frac{h}{y-x+h} = 1-\lambda$.
This implies that $\mu(x)-\mu(x+h) \geq \kappa \mu(x) - \kappa\mu(y+h) \geq \mu(y)-\mu(y+h)$
and the result follows from Theorem \ref{thm:gen_condition}.
\end{proof}

\begin{lemma}
The \L{}ukasiewicz $T$-norm $T_L$ satisfies property $A'$ and the bounded sum $T_L'$ satisfies property $A$. 
\end{lemma}
\begin{proof}
Clear from the definition of Eqns (\ref{eqn:propA}-\ref{eqn:propA'}).
\end{proof}

\begin{lemma}
Among the $T$-norms and $T$-conorms in Table \ref{tbl:t_norm} and Table \ref{tbl:t_conorm}, Table \ref{tbl:t_norm} indicates which $T$-norms satisfy properties $A$ and $B$ and Table \ref{tbl:t_conorm} indicates which $T$-conorms satisfy properties $A'$ and $B'$.
\label{lem:norms}
\end{lemma}
\begin{proof}

The $T$-norms $\min$, product and \L{}ukasiewicz satisfy property $B$ since they are copulas.
It is clear than $\min$ satisfies property $A$.

For the product $T$-norm $f(x,y) = xy$, let $\beta = w-z\geq 0$ and $\alpha = y-x\geq 0$.
If $w+x\leq y+z$, then $\beta \leq \alpha$ and $yz-xw = yz - (y-\alpha)(z+\beta) = \alpha z-\beta y + \alpha \beta \geq 0$, i.e. $f$ satisfies property $A$.

For the  \L{}ukasiewicz $T$-norm  $f(x,y) = \max(x+y-1,0)$, $w+x\leq y+z$ implies $f(w+x) \leq f(y+z)$ and $f$ satisfies property $A$.

For the nilpotent minimum, if $w+x > 1$, then $y+z\geq w+x > 1$ and $f(x,w) = \min(x,w) = x \leq y = \min(y,z) = f(y,z)$.
If $w+x \leq 1$, then $f(x,w) = 0\leq f(y,z)$. Thus property $A$ is satisfied. Note that since it not continuous, it does not satisfy property $B$.

The Dubois-Prade $T$-norms satisfy properties $A$ and $B$ by Lemma \ref{lem:ordinal_sum} since they are the ordinal sum of product and $\min$ $T$-norms.
The Mayor-Torrens $T$-norms satisfy properties $A$ and $B$ since they are the ordinal sum of \L{}ukasiewicz and $\min$ $T$-norms.

The Sugeno-Weber (for $\alpha \geq 0$), Clayton, Yager, Frank,  Ali-Mikhail-Haq, Gumbel and Joe $T$-norms are Archimedean copulas \cite{nelsen:copulas:2006,Kauers:Sugeno-Weber:2010,Klement2014} and thus satisfy properties $A$ and $B$ by Corollary \ref{cor:arch_copula}.

Since the $T$-conorm are duals to the $T$-norms, the conditions on the $T$-conorms are satisfied by Lemma \ref{lem:propAB_dual}.
\end{proof}

Our main results in this section show that the rearrangement inequality and its dual hold for many of the well-known $T$-norms and $T$-conorms listed in Tables \ref{tbl:t_norm}-\ref{tbl:t_conorm}. In particular,
the next results show that the (dual) rearrangement inequality holds for $T$-norms satisfying properties $A$ and $B$ and $T$-conorms satisfying $A'$ and $B'$.

\begin{theorem}
If $\otimes$ satisfies property  $B$ and $\oplus$ satisfies property $A'$, then the rearrangement inequality holds for $(\otimes, \oplus)$.
\label{thm:thm1}
\end{theorem}

\begin{proof} It suffices to show Eq. (\ref{eqn:condition}) is satisfied. Let $x_1\leq x_2$ and $y_1\leq y_2$. Without loss of generality, assume $x_1\otimes y_2 \leq x_2\otimes y_1$. Let $\alpha = x_1\otimes y_1$, $\beta = x_2\otimes y_1$, $\mu = x_1\otimes y_2 - x_1\otimes y_1 \geq 0$ and $\gamma =  x_2\otimes y_2 - x_2\otimes y_1\geq 0$. Since $\otimes$ satisfies property $B$, this implies that $\mu \leq \gamma$. Let $x=\alpha$, $y = \alpha+\mu$, $z = \beta$, $w=\beta+\gamma$.
Then $x+w\geq y+z$ and Eq. (\ref{eqn:condition}) can be rewritten as $x\oplus w \geq y\oplus z$ which is true as $\oplus$ satisfies property $A'$.
\end{proof}

\begin{theorem}
If $\otimes$ satisfies property  $A$ and $\oplus$ satisfies property $B'$, then the dual rearrangement inequality holds for $(\otimes, \oplus)$.
\label{thm:thm1_dual}
\end{theorem}
The proof is analogous to Theorem \ref{thm:thm1}. In order to show Eq. (\ref{eqn:condition_dual}), we assume $x_1\oplus y_2 \leq x_2\oplus y_1$ without loss of generality. Let $\alpha = x_1\oplus y_1$, $\beta = x_2\oplus y_1$, $\mu = x_1\oplus y_2 - x_1\oplus y_1 \geq 0$ and $\gamma =  x_2\oplus y_2 - x_2\oplus y_1\geq 0$. Since $\oplus$ satisfies property $B'$, this implies that $\mu \geq \gamma$. Let $x=\alpha$, $y = \alpha+\mu$, $z = \beta$, $w=\beta+\gamma$.
Then $x+w\leq y+z$ and Eq. (\ref{eqn:condition_dual}) can be rewritten as $x\otimes w \leq y\otimes z$ which is true as $\otimes$ satisfies property $A$.

\begin{corollary}
If $\otimes$ satisfies properties $A$ and $B$, then the rearrangement inequality and the dual rearrangement inequality hold for $(\otimes, \Phi (\otimes))$.
If $\oplus $ satisfies properties $A'$ and $B'$, then the rearrangement inequality and the dual rearrangement inequality hold for $(\Phi(\oplus), \oplus)$.
\end{corollary}

Since copulas satisfy property $B$ by definition, we have
\begin{theorem}
If the $T$-norm $\otimes$ is a copula  and $\oplus$ satisfies property $A'$ then  $(\otimes, \oplus)$ 
satisfies the rearrangement equality.
If  $\otimes$ satisfies  property $A$ and the $T$-norm $\oplus$ is the dual of a copula, then $(\otimes, \oplus)$ 
satisfies the dual rearrangement inequality. 
\label{thm:main_copulas}
\end{theorem}

The following follows from Corollary \ref{cor:arch_copula} and Theorems \ref{thm:thm1}-\ref{thm:thm1_dual}.
\begin{theorem} If $\otimes$ and $\Phi(\oplus)$ are both Archimedean copulas, then $(\otimes,\oplus)$ satisfies the rearrangement inequality and the dual rearrangement inequality. 
\label{thm:arch_rearrange}
\end{theorem}

The approach of decomposing Eqns. (\ref{eqn:condition}-\ref{eqn:condition_dual}) into property $A$ and $B$ does not always work as there are $T$-norms that are not copulas (e.g. drastic minimum which is not continuous) and thus does not satisfy condition $B$. The next result shows that even in that case Eqns. (\ref{eqn:condition}-\ref{eqn:condition_dual}) hold for some $T$-norms and their dual. 

\begin{theorem}
If $\otimes$ is minimum  and $\oplus$ is a uninorm, then 
$(\otimes,\oplus)$ satisfies the rearrangement inequality and the dual rearrangement inequality.
If $\oplus$ is maximum and $\otimes$ is a uninorm,
then $(\otimes,\oplus)$ satisfies the rearrangement inequality and the dual rearrangement inequality.
\label{thm:main_min}
\end{theorem}

\begin{proof}
Let $x_1\leq x_2$, $y_1\leq y_2$. 
Let $\otimes = \min$. Without loss of generality, assume $x_1\leq y_1$.
Then Eq. (\ref{eqn:condition}) can be written as 
$x_1\oplus \min(x_2,y_2) \geq x_1\oplus \min(x_2,y_1)$ which is true since $y_1\leq y_2$ and $\oplus$ is monotone.
Similarly, Eq. (\ref{eqn:condition_dual}) can be written as 
$\min(x_1\oplus y_1, x_2\oplus y_2) \leq \min(x_1\oplus y_2, x_2\oplus y_1)$
which is true since  $x_1\oplus y_1 \leq x_1\oplus y_2$ and  $x_1\oplus y_1 \leq x_2\oplus y_1$ by monotonicity of $\oplus$.

Suppose $\oplus = \max$.
Then Eq. (\ref{eqn:condition}) can be written as 
$\max(x_1\otimes y_1, x_2\otimes y_2) \geq \max(x_1\otimes y_2, x_2\otimes y_1)$ which is true as $x_2\otimes y_2$ is the largest of the $4$ terms.
Similarly, Eq. (\ref{eqn:condition_dual}) can be written as 
$y_1\otimes \max(x_2,y_2) \leq y_2 \otimes \max(x_2,y_1)$. 
If $x_2\leq y_2$, this is reduced to $y_1\otimes y_2 \leq y_2 \otimes \max(x_2,y_1)$ which is true as $\otimes$ is commutative and $\max(x_2,y_1)\geq y_1$.
If $x_2 \geq y_2$, this is reduced to $y_1\otimes x_2\leq y_2 \otimes x_2$ which again is true.
\end{proof}

\begin{theorem}
If $\otimes$ is drastic minimum and $\oplus$ is a disjunctive uninorm, then 
$(\otimes,\oplus)$ satisfies the rearrangement inequality.
If $\oplus$ is drastic maximum and $\otimes$ is a conjunctive uninorm,
then $(\otimes,\oplus)$ satisfies the dual rearrangement inequality.
\label{thm:main_dmin}
\end{theorem}

\begin{proof}
Let $x_1\leq x_2$, $y_1\leq y_2$. 
Let $\otimes$ be the drastic minimum $T_d$ and $\oplus$ be a disjunctive uninorm. If $x_2< 1$ and $y_2 < 1$, Eq. (\ref{eqn:condition}) becomes trivially $0\oplus 0 \geq 0\oplus 0$.
If $x_2=y_2=1$, 
we have $x_2\otimes y_2 = 1$ and Eq. (\ref{eqn:condition}) becomes $1\oplus (x_1\otimes y_1) \geq (x_1\otimes y_2)\oplus (x_2\otimes y_1)$ which is true since $1\oplus (x_1\otimes y_1)\geq 1\oplus 0 = 1$ as $\oplus$ is disjunctive.
If $x_2 = 1$ and $y_2 < 1$, then Eq. (\ref{eqn:condition}) becomes
$(x_1\otimes y_1) \oplus y_2 \geq (x_1\otimes y_2) \oplus y_1$. If $x_1 < 1$, this is reduced to $0\oplus y_2\geq 0 \oplus y_1$ which is true.
If $x_1 = 1$, this becomes $y_1\oplus y_2 \geq y_2\oplus y_1$ which is also true.
A symmetric argument shows the remaining case $x_2 < 1$ and $y_2 = 1$.

Let $\oplus$ be the drastic maximum $T_d'$ and $\otimes$ be a conjunctive uninorm. 
If $x_1 >  0$ and $y_1 >  0$, Eq. (\ref{eqn:condition_dual}) becomes trivially $1\otimes 1 \leq 1\otimes 1$.
If $x_1=y_1=0$, 
we have $x_1\oplus y_1 = 0$ and Eq. (\ref{eqn:condition_dual}) becomes $0\otimes (x_2\oplus y_2) \leq (x_1\oplus y_2)\otimes (x_2\oplus y_1)$ which is true since $0\otimes (x_2\oplus y_2) \leq 0\otimes 1 = 0$ as $\otimes$ is conjunctive.
If $x_1 = 0$ and $y_1 > 0$, then Eq. (\ref{eqn:condition_dual}) becomes
$y_1 \otimes (x_2\oplus y_2)  \leq y_2 \otimes (x_2\oplus y_1)$. If $x_2 > 0$, this is reduced to $y_1\otimes 1 \leq y_2\otimes 1$ which is true.
If $x_2 = 0$, this becomes $y_1\otimes y_2 \leq y_2\otimes y_1$ which is also true.
A symmetric argument shows the remaining case $x_1 >0$ and $y_1 = 0$.
\end{proof}

\begin{theorem}
If $\oplus$ is drastic maximum and $\otimes$ is a $T$-norm with no zero divisors,
then $(\otimes,\oplus)$ satisfies the rearrangement inequality.
If $\otimes$ is drastic minimum and $\oplus$ is a $T$-conorm such that $\Phi(\oplus)$ has no zero divisors, then 
$(\otimes,\oplus)$ satisfies the dual rearrangement inequality.
\label{thm:main_dmin2}
\end{theorem}
\begin{proof}
Let $x_1\leq x_2$, $y_1\leq y_2$, $\oplus = T_d'$ and $\otimes$ be a $T$-norm with no zero divisors.
Define $p = x_1\otimes y_1$, $q = x_2\otimes y_2$, $r = x_1\otimes y_2$, $s = x_2\otimes y_1$. Suppose $p = 0$, then as $\otimes$ has no zero divisors, $x_1 = 0$ or $y_1 = 0$.
Without loss of generality, we assume $x_1 = 0$. This implies that $r = 0$ and Eq. (\ref{eqn:condition}) is reduced to
$q \geq s$ which is true by monotonicity of $\oplus$.
If $p > 0$, then $q, r, s > 0$ by monotonicity and Eq. (\ref{eqn:condition}) can be written as $1\geq 1$.

Let $\otimes = T_d$ and $\Phi(\oplus)$ be a $T$-norm with no zero divisors.
Define $p = x_1\oplus y_1$, $q = x_2\oplus y_2$, $r = x_1\oplus y_2$, $s = x_2\oplus y_1$. Suppose $q = 1$, then as $\Phi(\oplus)$ has no zero divisors,  this implies $x_2 = 1$ or $y_2 = 1$.
Without loss of generality, we assume $x_2 = 1$. This implies that $s = 1$ and Eq. (\ref{eqn:condition_dual}) is reduced to
$p \leq r$ which is true by monotonicity of $\otimes$.
If $q < 1$, then $p, r, s < 1$ by monotonicity and Eq. (\ref{eqn:condition_dual}) can be written as $0\leq 0$.
\end{proof}

Examples of $T$-norms without zero divisors include the product $T$-norm, $T_{DP}^\alpha$, and Archimedean copulas such that the additive generator satisfies $\mu(0)=\infty$ (i.e. strict copulas \cite{Klement:positionII:2004}), e.g. $T_C^\alpha$ for $\alpha < 0$, $T_{AMH}^\alpha$, $T_F^\alpha$, $T_G^\alpha$ and $T_J^\alpha$.

\begin{theorem}
Both $(T_n,T_L)$ and $(T_L',T_n')$ satisfy the rearrangement inequality and its dual.
\label{thm:special_case}
\end{theorem}
\begin{proof}
Let $x_1\leq x_2$ and $y_1\leq y_2$, and $\otimes = T_n$ and $\oplus = T_L$, if
$x_1+y_1 < 1$, then $x_1\otimes y_1 = 0$ and $(x_1\otimes y_1)\oplus(x_2\otimes y_2) = 0$ since $T_L$ is a $T$-norm. On the other hand $(x_1\otimes y_2) \leq x_1$ and $(x_2\otimes y_1) \leq y_1$, and thus $(x_1\otimes y_2) +(x_2\otimes y_1) \leq x_1+y_1 < 1$ and $(x_1\otimes y_2) \oplus(x_2\otimes y_1) = 0$ and  Eq. (\ref{eqn:condition}) is satisfied.
If $x_1+y_1 \geq 1$, then Eq. (\ref{eqn:condition}) is reduced to
$\min(x_1,y_1)\oplus \min(x_2, y_2) \geq \min(x_1,y_2)\oplus \min(x_2,y_1)$. Without loss of generality, assume $x_1\leq y_1$.
Then we have $x_1 \oplus \min(x_2, y_2) \geq x_1 \oplus \min(x_2,y_1)$ which is true by the monotonicity of $\oplus$.

As for Eq. (\ref{eqn:condition_dual}), if $x_1+y_1 < 1$, then $x_1\oplus y_1 = 0$ and $(x_1\oplus y_1)\otimes (x_2\oplus y_2) = 0$. Thus Eq. (\ref{eqn:condition_dual}) is satisfied. 
If $x_1+y_1 \geq 1$, then $(x_1\oplus y_1)\otimes (x_2\oplus y_2) = (x_1+y_1-1)\otimes (x_2+y_2-1)$. If $x_1+y_1+x_2+y_2 < 3$, then this is equal to $0$ and Eq. (\ref{eqn:condition_dual}) is satisfied. If $x_1+y_1+x_2+y_2 \geq 3$, then $(x_1+y_1-1)\otimes (x_2+y_2-1)= \min(x_1+y_1-1,x_2+y_2-1) = \min(x_1+y_1,x_2+y_2)-1$.
Similarly  $(x_1\oplus y_2)\otimes (x_2\oplus y_1) = \min(x_1+y_2,x_2+y_1)-1$ and thus Eq. (\ref{eqn:condition_dual}) is satisfied as $x_1+y_1$ is the smallest term.

The case where $\otimes = T_L'$ and $\oplus = T_n'$ follows from applying Theorem \ref{thm:negation_RI} to the case $\otimes = T_n$ and $\oplus = T_L$.
\end{proof}

\begin{theorem}
$(T_n,T_n')$ satisfies the rearrangement inequality and its dual.
\label{thm:special_case2}
\end{theorem}
\begin{proof}
Let $x_1\leq x_2$ and $y_1\leq y_2$, and $\otimes = T_n$ and $\oplus = T_n'$, if
$x_1+y_1 < 1$, then $x_1\otimes y_1 = 0$ and $(x_1\otimes y_1)\oplus(x_2\otimes y_2) = x_2\otimes y_2$. On the other hand $(x_1\otimes y_2) \leq x_1$ and $(x_2\otimes y_1) \leq y_1$, and thus $(x_1\otimes y_2) +(x_2\otimes y_1) \leq x_1+y_1 < 1$ and $(x_1\otimes y_2) \oplus(x_2\otimes y_1) = \max((x_1\otimes y_2),(x_2\otimes y_1)) \leq (x_2\otimes y_2)$ and Eq. (\ref{eqn:condition}) is satisfied.
If $x_1+y_1 \geq 1$, the same proof as in Theorem \ref{thm:special_case} shows that Eq. (\ref{eqn:condition}) is satisfied. Thus $(T_n,T_n')$ satisfies the rearrangement inequality. Theorem \ref{thm:negation_RI} implies that the dual rearrangement inequality is satisfied as well.
\end{proof}

The results of applying Theorems \ref{thm:thm1}-\ref{thm:special_case2} to the $T$-norms and $T$-conorms in Tables \ref{tbl:t_norm}-\ref{tbl:t_conorm} are summarized in Table \ref{tbl:rearrange_t_norm}, where a $\checkmark$ indicates that the pair $(\otimes,\oplus)$ satisfies the rearrangement inequality and a $\Diamond$ indicates that $(\otimes,\oplus)$ satisfies the dual rearrangement inequality. Furthermore, for many entries in Table \ref{tbl:rearrange_t_norm} without $\checkmark$ or $\Diamond$ a corresponding example can be found to show that the rearrangement inequality or respectively its dual is not satisfied.

\begin{landscape}
\begin{table}[htbp]
\centering
\begin{tabular}{|c|c||c|c|c|c|c|c|c|c|c|c|c|c|c|c|c|}
\hline
&&\multicolumn{15}{|c|}{$\oplus$} \\ \hline
& & $T_m'$ & $T_p'$& $T_L'$ &  $T_d'$ & $T_n'$ &  $T_{DP}'^{\alpha'}$ & $T_{AMH}'^\alpha$ & $T_{C}'^{\alpha}$ &  $T_{F}'^{\alpha}$ &  $T_{Y}'^{\alpha}$ &  $T_{MT}'^{\alpha}$  & $T_{SW}'^{\alpha}$ & $T_{G}'^{\alpha}$& $T_{J}'^{\alpha}$ & $T_L$ \\ \hline\hline
\multirow{15}{0.9em}{$\otimes$}& $T_m$ & $\checkmark, \Diamond$&$\checkmark, \Diamond$&$\checkmark, \Diamond$&$\checkmark, \Diamond$&$\checkmark, \Diamond$&$\checkmark, \Diamond$&$\checkmark, \Diamond$&$\checkmark, \Diamond$ &$\checkmark, \Diamond$&$\checkmark, \Diamond$&$\checkmark, \Diamond$&$\checkmark, \Diamond$&$\checkmark, \Diamond$&$\checkmark, \Diamond$&$\checkmark, \Diamond$\\ \cline{2-17}
&$T_p$ &$\checkmark, \Diamond$&$\checkmark, \Diamond$&$\checkmark, \Diamond$&$\checkmark,\Diamond$&\checkmark&$\checkmark, \Diamond$&$\checkmark, \Diamond$&$\checkmark, \Diamond$ &$\checkmark, \Diamond$&$\checkmark, \Diamond$&$\checkmark, \Diamond$&$\checkmark, \Diamond$&$\checkmark, \Diamond$&$\checkmark, \Diamond$&$\checkmark$\\ \cline{2-17}
&$T_L$ &$\checkmark, \Diamond$&$\checkmark, \Diamond$&$\checkmark, \Diamond$&$\Diamond$&\checkmark&$\checkmark, \Diamond$&$\checkmark, \Diamond$&$\checkmark, \Diamond$ &$\checkmark, \Diamond$&$\checkmark, \Diamond$&$\checkmark, \Diamond$&$\checkmark, \Diamond$&$\checkmark, \Diamond$&$\checkmark, \Diamond$&$\checkmark$\\ \cline{2-17}
&$T_d$ &$\checkmark,\Diamond$&$\checkmark,\Diamond$&$\checkmark$&$\checkmark,\Diamond$&$\checkmark$&$\checkmark,\Diamond$&$\checkmark,\Diamond$&$\checkmark$&$\checkmark,\Diamond$&$\checkmark$&$\checkmark$&$\checkmark$&$\checkmark,\Diamond$&$\checkmark,\Diamond$&\\ \cline{2-17}
&$T_n$ &$\checkmark,\Diamond$&$\Diamond$&$\Diamond$&$\Diamond$&$\checkmark,\Diamond$&$\Diamond$&$\Diamond$&$\Diamond$ &$\Diamond$&$\Diamond$&$\Diamond$&$\Diamond$&$\Diamond$&$\Diamond$&$\checkmark,\Diamond$\\ \cline{2-17}
&$T_{DP}^{\alpha}$ &$\checkmark, \Diamond$&$\checkmark, \Diamond$&$\checkmark, \Diamond$&$\checkmark,\Diamond$&\checkmark&$\checkmark, \Diamond$&$\checkmark, \Diamond$&$\checkmark, \Diamond$ &$\checkmark, \Diamond$&$\checkmark, \Diamond$&$\checkmark, \Diamond$&$\checkmark, \Diamond$&$\checkmark, \Diamond$&$\checkmark, \Diamond$&$\checkmark$\\ \cline{2-17}
&$T_{AMH}^\alpha$ &$\checkmark, \Diamond$&$\checkmark, \Diamond$&$\checkmark, \Diamond$&$\checkmark,\Diamond$&\checkmark&$\checkmark, \Diamond$&$\checkmark, \Diamond$&$\checkmark, \Diamond$ &$\checkmark, \Diamond$&$\checkmark, \Diamond$&$\checkmark, \Diamond$&$\checkmark, \Diamond$&$\checkmark, \Diamond$&$\checkmark, \Diamond$&$\checkmark$\\ \cline{2-17}
&$T_{C}^\alpha$ &$\checkmark, \Diamond$&$\checkmark, \Diamond$&$\checkmark, \Diamond$&$\Diamond$&\checkmark&$\checkmark, \Diamond$&$\checkmark, \Diamond$&$\checkmark, \Diamond$ &$\checkmark, \Diamond$&$\checkmark, \Diamond$&$\checkmark, \Diamond$&$\checkmark, \Diamond$&$\checkmark, \Diamond$&$\checkmark, \Diamond$&$\checkmark$\\ \cline{2-17}
&$T_{F}^\alpha$ &$\checkmark, \Diamond$&$\checkmark, \Diamond$&$\checkmark, \Diamond$&$\checkmark, \Diamond$&\checkmark&$\checkmark, \Diamond$&$\checkmark, \Diamond$&$\checkmark, \Diamond$ &$\checkmark, \Diamond$&$\checkmark, \Diamond$&$\checkmark, \Diamond$&$\checkmark, \Diamond$&$\checkmark, \Diamond$&$\checkmark, \Diamond$&$\checkmark$\\ \cline{2-17}
&$T_{Y}^\alpha$ &$\checkmark, \Diamond$&$\checkmark, \Diamond$&$\checkmark, \Diamond$&$\Diamond$&\checkmark&$\checkmark, \Diamond$&$\checkmark, \Diamond$&$\checkmark, \Diamond$ &$\checkmark, \Diamond$&$\checkmark, \Diamond$&$\checkmark, \Diamond$&$\checkmark, \Diamond$&$\checkmark, \Diamond$&$\checkmark, \Diamond$&$\checkmark$\\ \cline{2-17}
&$T_{MT}^\alpha$ &$\checkmark, \Diamond$&$\checkmark, \Diamond$&$\checkmark, \Diamond$&$\Diamond$&\checkmark&$\checkmark, \Diamond$&$\checkmark, \Diamond$&$\checkmark, \Diamond$ &$\checkmark, \Diamond$&$\checkmark, \Diamond$&$\checkmark, \Diamond$&$\checkmark, \Diamond$&$\checkmark, \Diamond$&$\checkmark, \Diamond$&$\checkmark$\\ \cline{2-17}
&$T_{SW}^\alpha$ &$\checkmark, \Diamond$&$\checkmark, \Diamond$&$\checkmark, \Diamond$&$\Diamond$&\checkmark&$\checkmark, \Diamond$&$\checkmark, \Diamond$&$\checkmark, \Diamond$ &$\checkmark, \Diamond$&$\checkmark, \Diamond$&$\checkmark, \Diamond$&$\checkmark, \Diamond$&$\checkmark, \Diamond$&$\checkmark, \Diamond$&$\checkmark$\\ \cline{2-17}
&$T_{G}^\alpha$ &$\checkmark, \Diamond$&$\checkmark, \Diamond$&$\checkmark, \Diamond$&$\checkmark,\Diamond$&\checkmark&$\checkmark, \Diamond$&$\checkmark, \Diamond$&$\checkmark, \Diamond$ &$\checkmark, \Diamond$&$\checkmark, \Diamond$&$\checkmark, \Diamond$&$\checkmark, \Diamond$&$\checkmark, \Diamond$&$\checkmark, \Diamond$&$\checkmark$\\ \cline{2-17}
&$T_{J}^\alpha$ &$\checkmark, \Diamond$&$\checkmark, \Diamond$&$\checkmark, \Diamond$&$\checkmark,\Diamond$&\checkmark&$\checkmark, \Diamond$&$\checkmark, \Diamond$&$\checkmark, \Diamond$ &$\checkmark, \Diamond$&$\checkmark, \Diamond$&$\checkmark, \Diamond$&$\checkmark, \Diamond$&$\checkmark, \Diamond$&$\checkmark, \Diamond$&$\checkmark$\\ \cline{2-17}
&$T_{L}'$ &$\checkmark,\Diamond$&$\Diamond$&$\Diamond$&&$\checkmark,\Diamond$&$\Diamond$&$\Diamond$&$\Diamond$&$\Diamond$&$\Diamond$&$\Diamond$&$\Diamond$&$\Diamond$&$\Diamond$&\\ \hline
\end{tabular}
\caption{Pairs of $T$-norms and $T$-conorms $(\otimes, \oplus)$ for which the rearrangement inequality (or its dual) holds. A $\checkmark$ and $\Diamond$ indicates the rearrangement inequality and its dual holds respectively. For $\alpha < 0$, $(T_d,T_{C}'^\alpha)$ satisfies the dual rearrangement inequality and $(T_C^\alpha,T_d')$ satisfies the rearrangement inequality.}
\label{tbl:rearrange_t_norm}
\end{table}
\end{landscape}

\section{Variants of rearrangement inequality}
Several variations of rearrangement inequalities are studied in Ref. \cite{wu:rearrangement:2022} that only rely on satisfying equations of the form Eqns (\ref{eqn:condition})-(\ref{eqn:condition_dual}) and thus Theorems \ref{thm:thm1}-\ref{thm:special_case2} hold for these variants as well. As specific examples, we describe how these results apply to Table \ref{tbl:rearrange_t_norm}.
 
\begin{theorem} \label{thm:variation-sumprod}
Let $a_i$ be a set of $2n$ numbers in $[0,1]$ and let $b_i$ be the numbers $a_i$ sorted such that $b_1 \leq b_2\leq \cdots \leq b_{2n}$. 
If the pair $(\otimes,\oplus)$ corresponds to a $\checkmark$ in Table \ref{tbl:rearrange_t_norm}, then
\begin{equation}\bigoplus_{i=1}^n b_{i}\otimes b_{2n-i+1} \leq   \bigoplus_{i=1}^n a_{2i-1}\otimes a_{2i} \leq \bigoplus_{i=1}^n b_{2i-1}\otimes b_{2i}.\label{eqn:variant1} \end{equation}
If the pair $(\otimes,\oplus)$ corresponds to a $\Diamond$ in Table \ref{tbl:rearrange_t_norm}, then
\begin{equation}\bigotimes_{i=1}^n \left(b_{2i-1}\oplus b_{2i}\right) \leq   \bigotimes_{i=1}^n \left(a_{2i-1}\oplus a_{2i}\right) \leq \bigotimes_{i=1}^n \left(b_{i}\oplus b_{2n-i+1}\right).\label{eqn:variant1_dual} \end{equation}
\end{theorem}
 
\subsection{Circular rearrangement inequality}
This is a variant of rearrangement inequalities considered in Ref. \cite{Yu2018}. 
Let $\sigma_{m_1}$ denote the permutation $(1,n-1,3,n-3,5,\cdots, n-6,6,n-4,4,n-2,2,n)$ and
$\sigma_{m_2}$ denote the permutation $(1,3,5,\cdots, n, \cdots, 6,4,2)$.

\begin{theorem} \label{thm:circular}
Let $0\leq a_1 \leq a_2\leq \cdots \leq a_{n}\leq 1$ and $\sigma \in S_n$ be a permutation on $\{1,\cdots n\}$.
If the pair $(\otimes,\oplus)$ corresponds to a $\checkmark$ in Table \ref{tbl:rearrange_t_norm}, then
\begin{equation} (a_{\sigma(1)}\otimes a_{\sigma(2)}) \oplus  (a_{\sigma(2)}\otimes a_{\sigma(3)})  \oplus  \cdots \oplus (a_{\sigma(n)}\otimes a_{\sigma(1)}) \label{eqn:circular} \end{equation}
is minimized and maximized when the permutation $\sigma$ is equal to $\sigma_{m_1}$ and $\sigma_{m_2}$ respectively.
If the pair $(\otimes,\oplus)$ corresponds to a $\Diamond$ in Table \ref{tbl:rearrange_t_norm}, then
\begin{equation} (a_{\sigma(1)}\oplus a_{\sigma(2)}) \otimes  (a_{\sigma(2)}\oplus a_{\sigma(3)})  \otimes  \cdots \otimes (a_{\sigma(n)}\oplus a_{\sigma(1)}) \label{eqn:circular_dual} \end{equation}
is maximized and minimized when the permutation $\sigma$ is equal to $\sigma_{m_1}$ and $\sigma_{m_2}$ respectively.
\end{theorem}

\section{Conclusions}
We derived conditions under which $T$-norm based multivalued logics satisfy the rearrangement inequality and its dual. In particular, these inequalities are satisfied when the $T$-norms are Archimedean copulas, which covers many well-known classes of $T$-norms in the literature. On the other hand, the definitions of properties $A$, $B$, $A'$ and $B'$ are not restricted to $T$-norms and $T$-conorms and can be applied to any nonnegative functions that satisfies the commutativity, associativity and monotonicity conditions to show that the rearrangement inequality and its dual hold for them.


\begin{thebibliography}{10}

\bibitem{Hardy1952}
G.~H. Hardy, J.~E. Littlewood, and G.~P{\'o}lya, {\em Inequalities}.
\newblock Cambridge university press, 1952.

\bibitem{oppenheim:rearrangement:1954}
A.~Oppenheim, ``Inequalities connected with definite {H}ermitian forms, {II},''
  {\em American Mathematical Monthly}, vol.~61, pp.~463--466, 1954.

\bibitem{Minc1971}
H.~Minc, ``Rearrangements,'' {\em Transactions of the American Mathematical
  Society}, vol.~159, pp.~497--504, 1971.

\bibitem{wu:rearrangement:2022}
C.~W. Wu, ``On rearrangement inequalities for multiple sequences,'' {\em
  Mathematical Inequalities \& Applications}, vol.~25, no.~2, pp.~511--534,
  2022.

\bibitem{sep-logic-manyvalued}
S.~Gottwald, ``{Many-Valued Logic},'' in {\em The {Stanford} Encyclopedia of
  Philosophy} (E.~N. Zalta, ed.), Metaphysics Research Lab, Stanford
  University, {S}ummer 2020~ed., 2020.

\bibitem{fodor:t-norm:2004}
J.~Fodor, ``Left-continuous t-norms in fuzzy logic: an overview,'' {\em Acta
  Polytechnica Hungarica}, vol.~1, no.~2, 2004.

\bibitem{KLEMENT:positonpaperI:2004}
E.~P. Klement, R.~Mesiar, and E.~Pap, ``Triangular norms. position paper {I}:
  basic analytical and algebraic properties,'' {\em Fuzzy Sets and Systems},
  vol.~143, no.~1, pp.~5--26, 2004.
\newblock Advances in Fuzzy Logic.

\bibitem{JENEI200427}
S.~Jenei, ``How to construct left-continuous triangular norms—state of the
  art,'' {\em Fuzzy Sets and Systems}, vol.~143, no.~1, pp.~27--45, 2004.
\newblock Advances in Fuzzy Logic.

\bibitem{trillas:negation:1979}
E.~Trillas, ``Sobre functiones de negaci\'{o}n en la teor\'{i}a de conjuntas
  difusos,'' {\em Stochastica}, vol.~3, pp.~47--60, 1979.

\bibitem{dubois1980fuzzy}
D.~Dubois and H.~Prade, {\em Fuzzy sets and systems: theory and applications},
  vol.~144.
\newblock Academic press, 1980.

\bibitem{Mayor1991}
G.~Mayor and J.~Torrens, ``On a family of t-norms,'' {\em Fuzzy Sets and
  Systems}, vol.~41, no.~2, pp.~161--166, 1991.

\bibitem{Klement:positionII:2004}
E.~P. Klement, R.~Mesiar, and E.~Pap, ``Triangular norms. {Position} paper
  {II}: {General} constructions and parameterized families,'' {\em Fuzzy Sets
  and Systems}, vol.~145, no.~3, pp.~411--438, 2004.

\bibitem{nelsen:copulas:2006}
R.~B. Nelsen, {\em An Introduction to Copulas}.
\newblock Springer, 2006.

\bibitem{clifford:semigroups:1954}
A.~H. Clifford, ``Naturally totally ordered commutative semigroups,'' {\em
  American Journal of Mathematics}, vol.~76, pp.~631--646, July 1954.

\bibitem{JENEI2002199}
S.~Jenei, ``A note on the ordinal sum theorem and its consequence for the
  construction of triangular norms,'' {\em Fuzzy Sets and Systems}, vol.~126,
  no.~2, pp.~199--205, 2002.

\bibitem{YAGER1996111}
R.~R. Yager and A.~Rybalov, ``Uninorm aggregation operators,'' {\em Fuzzy Sets
  and Systems}, vol.~80, no.~1, pp.~111--120, 1996.

\bibitem{fodor:uninorm:1997}
J.~C. Fodor, R.~R. Yager, and A.~Rybalov, ``Structure of uninorms,'' {\em
  International Journal of Uncertainty, Fuzziness and Knowledge-Based Systems},
  vol.~5, no.~4, pp.~411--427, 1997.

\bibitem{janssens:t-norm:2004}
S.~Janssens, B.~de~Baets, and H.~de~Meyer, ``Bell-type inequalities for
  parametric families of triangular norms,'' {\em Kybernetika}, vol.~40, no.~1,
  pp.~89--106, 2004.

\bibitem{moynihan:copula:1978}
R.~Moynihan, ``On $\tau_{T}$ semigroups of probability distribution functions
  {II},'' {\em Aequationes mathematicae}, vol.~17, pp.~19--40, 1978.

\bibitem{BACIGAL:additive:2015}
T.~Bacigál, V.~Najjari, R.~Mesiar, and H.~Bal, ``Additive generators of
  copulas,'' {\em Fuzzy Sets and Systems}, vol.~264, pp.~42--50, 2015.
\newblock Special issue on Aggregation functions at AGOP2013 and EUSFLAT 2013.

\bibitem{Mesiar:ordinalsums:2010}
R.~Mesiar and C.~Sempi, ``Ordinal sums and idempotents of copulas,'' {\em
  Aequationes Mathematicae}, vol.~79, no.~1-2, pp.~39--52, 2010.

\bibitem{KLEMENT:pseudo-inverse:1999}
E.~P. Klement, R.~Mesiar, and E.~Pap, ``Quasi- and pseudo-inverses of monotone
  functions, and the construction of t-norms,'' {\em Fuzzy Sets and Systems},
  vol.~104, no.~1, pp.~3--13, 1999.

\bibitem{Kauers:Sugeno-Weber:2010}
M.~Kauers, V.~Pillwein, and S.~Saminger-Platz, ``Dominance in the family of
  {S}ugeno-{W}eber t-norms,'' {\em Fuzzy Sets and Systems}, vol.~181, 07 2010.

\bibitem{Klement2014}
E.~Klement, R.~Mesiar, F.~Spizzichino, and A.~Stupnanov{\'{a}}, ``Universal
  integrals based on copulas,'' {\em Fuzzy Optim. Decis. Mak.}, vol.~13, no.~3,
  pp.~273--286, 2014.

\bibitem{Yu2018}
H.~{Yu}, ``Circular rearrangement inequality,'' {\em Journal of Mathematical
  Inequalities}, vol.~12, no.~3, pp.~635--643, 2018.

\end{thebibliography}
\end{document}